\documentclass[noinfoline]{article}

\RequirePackage[
amsthm,amsmath
]{imsart}

\usepackage{graphicx}
\usepackage{latexsym,amsmath}
\usepackage{amsmath,amsthm,amscd}
\usepackage{amsfonts}
\usepackage[psamsfonts]{amssymb}
\usepackage{enumerate}

\usepackage{url}

\usepackage{tocvsec2}

\usepackage{epstopdf}
\usepackage{wrapfig}
\usepackage{float}

\startlocaldefs
\numberwithin{equation}{section}

\theoremstyle{plain} 
\newtheorem{theorem}{Theorem}[section]
\newtheorem{corollary}[theorem]{Corollary}

\newtheorem{lemma}[theorem]{Lemma}

\newtheorem{proposition}[theorem]{Proposition}

\theoremstyle{definition} 

\theoremstyle{definition} 

\newtheorem*{ex*}{Example}
\theoremstyle{remark} 

\theoremstyle{remark} 

\newtheorem*{remark*}{Remark}
\numberwithin{equation}{section}
 
\newcommand{\beqa}{\begin{eqnarray}}
\newcommand{\eeqa}{\end{eqnarray}}
 
\newcommand{\bseq}{\begin{subequations}}
 
\newcommand{\eseq}{\end{subequations}}

\newcommand{\dd}{\partial}

\newcommand{\tv}{{\,\operatorname{TV}}}
\newcommand{\ko}{{\,\operatorname{Ko}}}

\newcommand{\sd}{{\,\operatorname{SD}_p}}
\newcommand{\sdq}{{\,\operatorname{SD}_q}}

\renewcommand{\dd}{{\,\operatorname{d}}}

\newcommand{\Ga}{\Gamma}

\newcommand{\la}{\lambda}

\newcommand{\de}{\delta}

\newcommand{\vpi}{\varphi}
\renewcommand{\th}{\theta}
\renewcommand{\Psi}{\overline{\Phi}}

\newcommand{\pd}[2]{\frac{\partial#1}{\partial#2}}

\renewcommand{\P}{\operatorname{\mathsf{P}}} 

\newcommand{\R}{\mathbb{R}}

\newcommand{\tH}{{\tilde{H}}}

\newcommand{\tx}{{\tilde{x}}}

\newcommand{\tP}{{\tilde{P}}}

\renewcommand{\le}{\leqslant}
\renewcommand{\ge}{\geqslant}

\endlocaldefs

\begin{document}

\begin{frontmatter}

\title{Exact bounds on the closeness between the Student and standard normal distributions}
\runtitle{Exact bounds on Student's distribution}

%

\begin{aug}
\author{\fnms{Iosif} \snm{Pinelis}\thanksref{t2}\ead[label=e1]{ipinelis@mtu.edu}}
  \thankstext{t2}{Supported by NSF grant DMS-0805946}
\runauthor{Iosif Pinelis}


\address{Department of Mathematical Sciences\\
Michigan Technological University\\
Houghton, Michigan 49931, USA\\
E-mail: \printead[ipinelis@mtu.edu]{e1}}
\end{aug}

\begin{abstract}
Upper bounds on the Kolmogorov distance (and, equivalently in this case, on the total variation distance) between the Student distribution with $p$ degrees of freedom ($\sd$) and the standard normal distribution are obtained. 
These bounds are in a certain sense best possible, and the corresponding relative errors are small even for moderate values of $p$. 
The same bounds hold on the closeness 
between $\sd$ and $\sdq$ with $q>p$. 
\end{abstract}




\end{frontmatter}

\settocdepth{chapter}

\tableofcontents 

\settocdepth{subsubsection}

\theoremstyle{plain} 
\numberwithin{equation}{section}


\section{Summary and discussion}\label{intro} 
The density and distribution functions of Student's distribution with $p$ degrees of freedom ($\sd$) are given, respectively, by the formulas 
\begin{align}
	f_p(x)&:=\frac{ \Gamma
   \left(\frac{p+1}{2}\right)}{\sqrt{\pi p}\, \Gamma \left(\frac{p}{2}\right)}\,
   \left(1+\frac{x^2}{p}\right)^{-(p+1)/2} \quad\text{and} \label{eq:f_p}\\
   F_p(x)&:=\int_{-\infty}^x f_p(u)\dd u \label{eq:F_p}
\end{align}
for all real $x$. 
Most often, the values of the parameter $p$ are assumed to be positive integers. 
However, formula \eqref{eq:f_p} defines a probability density function for all real $p>0$, and, as we shall see, it may be advantageous, at least as far as proofs are concerned, to let $p$ take on all positive real values.  
Let us also extend definitions \eqref{eq:f_p} and \eqref{eq:F_p} by continuity to $p=\infty$, so that 
\begin{equation*}
\text{$f_\infty=:\vpi$ and $F_\infty=:\Phi$ }	
\end{equation*}
are the density and distribution functions of the standard normal distribution (SND). 

The standard normal and Student distributions are clearly among the most common distributions in statistics.  
It is a textbook fact that the $\sd$ is close to the SND when $p$ is large, say in the sense that $f_p(x)\underset{p\to\infty}\longrightarrow f_\infty(x)$ for each real $x$. 
By Scheff\'e's theorem \cite{scheffe}, this implies the convergence of the total variation distance 
\begin{equation}\label{eq:d_TV}
	d_\tv(p)=\frac12\,\int_{-\infty}^\infty|f_p(x)-\vpi(x)|\dd x
\end{equation}
to $0$ as $p\to\infty$. 
In fact, the convergence of the $\sd$ to the SND is presented in \cite{scheffe} as the motivating case. 

Consider also the Kolmogorov distance 
\begin{equation*}
	d_\ko(p):=\sup_{x\in\R}|F_p(x)-\Phi(x)|
\end{equation*}
between the $\sd$ and SND. It is clear that, for any two probability distributions, the Kolmogorov distance between them is no greater than twice the total variation distance, and hence the convergence of the latter distance to $0$ implies that of the former. 

However, in the present case one can say more. 
For any $p$ and $q$ in the interval $(0,\infty]$, let $d_\ko(p,q)$ and $d_\tv(p,q)$ denote, respectively, the Kolmogorov distance and the total variation distance between $\sd$ and $\sdq$, so that $d_\ko(p)=d_\ko(p,\infty)$ and $d_\tv(p)=d_\tv(p,\infty)$. 

\begin{proposition}\label{prop:} \ 
\begin{enumerate}[(i)]
	\item 
	For all $p$ and $q$ such that $0<p<q\le\infty$
\begin{equation}\label{eq:ko,tv}
	\tfrac12d_\tv(p,q)=d_\ko(p,q)=\max_{x\in(0,\infty)}\big(F_q(x)-F_p(x)\big). 
\end{equation}
	\item
Moreover, for each $p\in(0,\infty)$ the distance $d_\ko(p,q)$ is strictly increasing in $q\in[p,\infty]$, and for each $q\in(0,\infty]$ the distance $d_\ko(p,q)$ is strictly decreasing in $p\in(0,q]$. 
In particular, 
\begin{equation}\label{eq:d(p,q)<d(p)}
d_\ko(p,q)<d_\ko(p,\infty)=d_\ko(p)	
\end{equation}
for all $p$ and $q$ such that $0<p\le q<\infty$, and $d_\ko(p)$ is strictly decreasing in $p\in(0,\infty]$. 
\end{enumerate}
Statement (ii) holds as well with $d_\tv$ in place of $d_\ko$. 
\end{proposition}

This proposition and the other results stated in this section will be proved in Section~\ref{proofs}. 

The Kolmogorov distance and the total variation one are apparently the two most commonly used distances between probability distributions. 
Therefore, it seems natural to consider the rate of convergence of $d_\ko(p)$ and, equivalently, $d_\tv(p)$ to $0$ as $p\to\infty$, which is part of what is done in this paper. 
Actually, the motivation for this study comes from the discussion in \cite{BE-student}. 
In turn, the paper \cite{BE-student} was motivated by developments of \cite{nonlinear}. 

\begin{theorem}\label{th:1/4}
For any real $p\ge4$ 
\begin{equation}\label{eq:<C/p}
\tfrac12\,d_\tv(p)=d_\ko(p)<C/p, 
\end{equation}
where 
\begin{equation}\label{eq:C}
	C:=
	\frac14\sqrt{\frac{7+5 \sqrt{2}}{\pi e^{1+\sqrt2}}}
	=0.158\dots.  
\end{equation}
Moreover, 
\begin{equation}\label{eq:lim=C}
	\lim_{p\to\infty}p\,d_\ko(p)=C,
\end{equation}
so that the constant $C$ is the best possible one in \eqref{eq:<C/p}. 
\end{theorem}

%

In what follows, it is assumed by default that 
\begin{equation*}
	a:=1/p. 
\end{equation*}
Theorem~\ref{th:1/4} is based on 

\begin{theorem}\label{th:bound}
For any real $p\ge\frac{50}{29}$ 
\begin{equation}\label{eq:<B}
\tfrac12\,d_\tv(p)=d_\ko(p)<B(a,\tx_a),
\end{equation}
where 
\begin{align*}
B(a,x):=\frac a{768} \Big( & 8 x \big[5 a^2 x^2+a (3 x^6-7 x^4-5 x^2-3)+24(x^2+1)\big] \vpi(x) \\
&+33 a^2 (2 \Phi(x)-1)\Big) 
\end{align*}
and $\tx_a$ is, for any $a\in(0,1)$, the unique real root $x>0$ of the polynomial equation 
\begin{multline}\label{eq:P}
	P(a,x):=-96 (x^4-2 x^2-1)-4 a (3 x^8-28 x^6+30 x^4+12 x^2+3)\\
   -a^2 (20 x^4-60 x^2-33)=0. 
\end{multline}
\end{theorem}
\noindent In fact, it will be shown (Lemma~\ref{lem:B<C/p}) that 
\begin{equation}\label{eq:B<C/p}
	B(a,\tx_a)<C/p
\end{equation}
for all $p\ge4$. 

Note that, since the polynomial equation \eqref{eq:P} is of degree $4$ in $x^2$, the root $\tx_a$ can be expressed in radicals of polynomials in $a$. 

By the triangle inequality, \eqref{eq:<C/p} implies that $\tfrac12\,d_\tv(p,q)=d_\ko(p,q)\le d_\ko(p)+d_\ko(q)<C/p+C/q$ for any real $p$ and $q$ that are no less than $4$. 
Taking \eqref{eq:ko,tv} and \eqref{eq:d(p,q)<d(p)} 
into account, one sees that \eqref{eq:<B} and \eqref{eq:B<C/p} immediately yield better bounds: 

\begin{corollary}\label{cor:} 
For all $p$ and $q$ such that $4\le p<q\le\infty$
\begin{equation}\label{eq:p,q}
	\tfrac12d_\tv(p,q)=d_\ko(p,q)<B(a,\tx_a)<C/p. 
\end{equation}
\end{corollary}

Graphs of the bounds $B(a,\tx_a)$ 
and $C/p$ 
are shown in Figure~\ref{fig:bounds-compare}, along with the corresponding graph of $d_\ko(p)$. 
This is done for the values of $p\in[\frac{50}{29},30]$, even though the upper bound $C/p$ on $d_\ko(p)$ has been established only for $p\ge4$. 

\begin{figure}[H] 
	\centering
\includegraphics[width=0.70\textwidth]{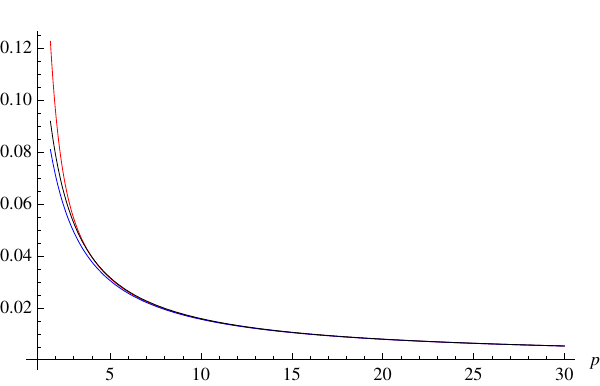
}
	\caption{Bounds $B(a,\tx_a)$ (red) and $C/p$ (blue), compared with $d_\ko(p)$ (black).} 
	\label{fig:bounds-compare}
\end{figure}

The relative errors $\frac{B(a,\tx_a)}{d_\ko(p)}-1$ and $\frac{C/p}{d_\ko(p)}-1$ of the bounds in \eqref {eq:<B} and \eqref {eq:<C/p} are shown in Figure~\ref{fig:rel-errs}, for $p\in[1,3.95]$ in the leftmost panel, for $p\in[3.95,4.05]$ in the middle panel, and for $p\in[4.05,30]$ in the rightmost one. 

\begin{figure}[H] 
	\centering		
	\includegraphics[width=1.00\textwidth]{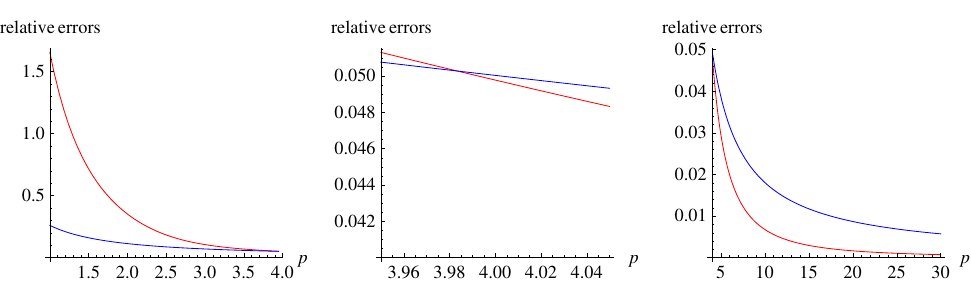
	}
	\caption{Relative errors $\frac{B(a,\tx_a)}{d_\ko(p)}-1$ (red) and $\frac{C/p}{d_\ko(p)}-1$ (blue) of the bounds in \eqref {eq:<B} and \eqref {eq:<C/p}.} 
	\label{fig:rel-errs}
\end{figure}

It appears that the bound $C/p$ would be more accurate than $B(a,\tx_a)$ for $p\in[1,3.98]$; remember, however, that the bound $B(a,\tx_a)$ was established only for $p\ge\frac{50}{29}$, from which the bound $C/p$ was deduced only for $p\ge4$. 
Anyway, the smaller values of $p>0$ may be of lesser interest, since for such $p$ the Student distribution is not very close to the standard normal one. 
On the other hand, for large $p$ the bound $B(a,\tx_a)$ appears significantly more accurate (in terms of the relative errors) --- but much more complicated --- than the bound $C/p$. 
Yet, even for $p$ as small as $4$, the relative errors of the bounds $C/p$ and $B(a,\tx_a)$ 
are both only about $5\%$, with the corresponding absolute errors less than $2\times10^{-3}$. 
For $p=12$, the relative and absolute errors of the bound $C/p$ are less than $1.5\%$ and $2\times10^{-4}$, 
respectively, and the corresponding figures for the bound $B(a,\tx_a)$ are about $0.5\%$ and $6\times10^{-5}$. 
%
%
Also, by \eqref{eq:lim=C}, the relative error $\frac{C/p}{d_\ko(p)}-1$ of the upper bound $C/p$ goes to $0$ as $p\to\infty$; in view of \eqref{eq:B<C/p}, the same holds for the upper bound $B(a,\tx_a)$. 
One may as well note that, if the distance $d_\ko(p)$ is considered as a kind of ``initial'' error --- of the approximation of the Student distribution by the SND, then the relative error $\frac{C/p}{d_\ko(p)}-1$ is a  relative error ``of the second order'', in the sense that it is the relative error of the estimate $C/p$ of the initial error $d_\ko(p)$; the same statement holds with $B(a,\tx_a)$ in place of $C/p$. 

Figure~\ref{fig:rel-errs} also suggests that the threshold value $4$ in the condition $p\ge4$ 
in Theorem~\ref{th:1/4} 
is very close to the best possible one for which the comparison \eqref{eq:B<C/p} between the bounds in \eqref{eq:<C/p} and \eqref{eq:<B} is still valid. 
%



\newcommand{\cpu}{\operatorname{CPU}}

The bounds in \eqref{eq:<C/p} and \eqref{eq:<B} may be compared with those obtained by 
Cacoullos, Papathanasiou and Utev \cite[Example~1, page~1614]{utev-cacoullos}, who used Stein-type methods to show that 
\begin{equation}
\tfrac12\,d_\tv(p)\le B_{p,\cpu}:=\tfrac4{p-2} 	
\end{equation}
for $p>2$. Figure~\ref{fig:CPU} suggests that the bounds in \eqref{eq:<C/p} and \eqref{eq:<B} are much smaller than $B_{p,\cpu}$.

\begin{figure}[H] 
	\centering
\includegraphics[width=0.70\textwidth]{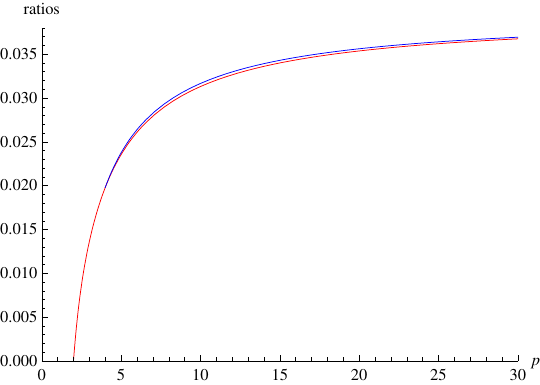}
	\caption{Ratios of the bounds $B(a,\tx_a)$ (red) and $C/p$ (blue) to $B_{p,\cpu}$.} 
	\label{fig:CPU}
\end{figure}
 
It was also shown in \cite{utev-cacoullos} that the total variation distance between $\sd$ and the centered normal distribution $N\big(0,p/(p-2)\big)$ with variance $p/(p-2)$ is no greater than $4/(p-1)$, again for $p>2$. 
For $p>4$ and the Kolmogorov distance between $\sd$ and $N\big(0,p/(p-2)\big)$ -- which is of course half the corresponding total variation distance, Shimizu \cite[(4.5)]{shimizu} obtained an upper bound, which is no less than (and asymptotic to, for $p\to\infty$) $C_1/(p-p_0)$, where $C_1:=1/\pi=0.31831\dots$ and $p_0:=(54 + 8 \sqrt2)/17=3.841\ldots$. 
As Table~1 in \cite{shimizu} suggests, these bounds in \cite{shimizu,utev-cacoullos} concerning the closeness of $\sd$ to $N\big(0,p/(p-2)\big)$ are not asymptotically optimal for large $p$, in contrast with the bounds in \eqref{eq:<C/p} and \eqref{eq:<B}. It appears likely that methods similar to the ones used in this paper can yield bounds with the asymptotically best possible constant factors for $N\big(0,p/(p-2)\big)$ as well.  
At this point, one may also note that, according to \cite{normal-scale}, the total variation distance between $N\big(0,p/(p-2)\big)$ and the SND is less than $C_2/\sqrt{p(p-2)}\sim C_2/p$, where $C_2:=\sqrt{2/(\pi e)}=0.48\dots$. 
The bounds in each of the papers \cite{shimizu,utev-cacoullos} were obtained by quite different methods and as corollaries of more general results. 

Also, upper bounds  
of the form $c/\sqrt n$ 
on the Kolmogorov and total variation distances between the distribution of the self-normalised sum (say $V_n$) of i.i.d.\ standard normal r.v.'s $Z_1,\dots,Z_n$ and the standard normal distribution were recently 
obtained in \cite{bourguin} by means of the Malliavin calculus, where $c$ is a positive absolute constant. 
However, the optimal bounds in this special ``i.i.d.\ standard normal'' case should be $O(1/n)$; indeed, 
in view of Theorem~\ref{th:1/4} above and \cite[Proposition~1.4]{BE-student}, 
$|\P(V_n\le z)-\Phi(z)|\le\frac{0.322}{n-1}$ for $n=2,3,\dots$ and all real $z$. 

In \cite{pink-wilk}, an asymptotic expansion for the tail $1-F_p(x)$ of the $\sd$ was obtained, which provides  successive approximations (say $A_{p,j}(x)$) that are good for very large values of $x$, as illustrated in Figure~\ref{fig:wilk} --- for $p=14$.  
\begin{figure}[htbp]
	\centering	\includegraphics[width=1.00\textwidth]{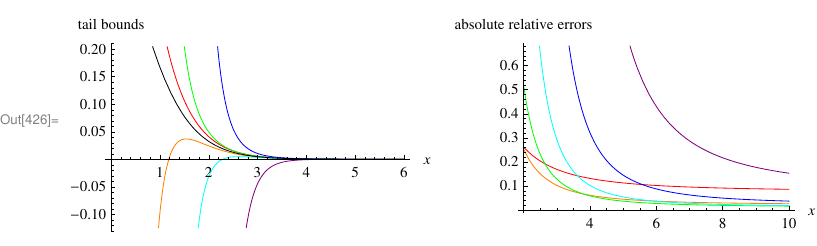
	}
	\caption{Left panel: the successive approximations $A_{14,1}(x),\dots,A_{14,6}(x)$ of $1-F_{14}(x)$ as in \cite{pink-wilk}, colored red, orange, green, cyan, blue, purple, respectively; the graph of the tail $1-F_{14}(x)$ is black. 
	Right panel: the graphs of the corresponding absolute relative errors $|\frac{A_{14,j}(x)}{1-F_{14}(x)}-1|$ ($j=1,\dots,6$).} 
	\label{fig:wilk}
\end{figure}
The right panel of Figure~\ref{fig:wilk} suggests that each approximation $A_{p,j}(x)$ has its own ``maximum competency'' zone of values of $x$, for which it is the best, over all $j$'s; 
it appears that 
this zone is a neighborhood of $\infty$, which gets narrower as $j$ increases. 
Clearly, the bounds given in the present paper differ quite significantly in kind and purpose from those given in \cite{pink-wilk}. 
%

\section{Proofs}\label{proofs} 

The main idea of the proof of the inequality in \eqref{eq:<C/p} is to reduce it, through a number of steps, to systems of algebraic inequalities. 
Such systems, 
by a well-known result of Tarski \cite{tarski48,collins98} (rooted in Sturm's
theorem), can be solved in a completely algorithmic manner. 
Similar results hold for certain other systems which may also involve the logarithmic function (whose derivative is algebraic), the SND density function $\vpi$ (whose logarithm is algebraic), and the SND distribution function $\Phi$ (whose derivative is $\vpi$). 
The bound $B(a,\tx_a)$ in Theorem~\ref{th:bound} is such an expression. 
The Tarski algorithm is implemented in latter versions of Mathematica via \texttt{Reduce} and other related commands. 
For instance, a command of the form 
$$
\text{
\texttt{Reduce[cond1 \&\& cond2 \&\& $\dots$, \{var1,var2,$\dots$,\}, Reals]}
}
$$
returns a simplified equivalent of the given system (of equations and/or inequalities) \texttt{cond1}, \texttt{cond2}, $\dots$ over real variables \texttt{var1}, \texttt{var2}, $\dots$.
However, the execution of such a command may take a very long time (and/or require too much computer memory) if the given system is more than a little complicated, as is e.g.\ the case with the system $B(a,\tx_a)<C/p\ \&\ a=1/p\ \&\ p\ge4$, which provides the way to deduce the bound in \eqref{eq:<C/p} from that in \eqref{eq:<B}. 
Therefore, Mathematica will need some human guidance here. 
It appears that all such 
calculations done with the help of a computer are, at least, as reliable and rigorous as the same calculations done only by hand. 

The main difficulty to overcome in this paper was to construct the upper bound $B(a,\tx_a)$ on $d_\ko(p)$, which would be, on the one hand, accurate enough and, on the other hand, provide a traversable bridge from $d_\ko(p)$ to the simple upper bound $C/p$, as indicated above. 
In turn, the bound $B(a,\tx_a)$ was obtained in several steps, described in detail in the statements of Lemmas~\ref{lem:x_p}--\ref{lem:xxa}, presented in Subsection~\ref{lemmas}. 

The first step is to note that the difference $f_\infty(x)-f_p(x)$ between the densities of the SND and $\sd$ changes its sign exactly once, from $+$ to $-$, as $x$ increases from $0$ to $\infty$ (Lemma~\ref{lem:x_p}). 
A key observation here (essentially borrowed from \cite{mono-student}) is that, luckily, for the defined in \eqref{eq:r_p} ratio $r_p(x)$ of the densities, the logarithmic partial derivarive $\pd{}p\ln r_p(x)$ 
increases in $x\in[0,1]$ and decreases in $x\in[1,\infty)$ --- 
with the same switch-point $1$ for all $p>0$. 
This implies that $r_p(x)$ decreases in $x\in[0,1]$ from $r_p(0)<1$ and then increases in $x\in[1,\infty)$ to $\infty$, so that  
the difference $F_\infty-F_p$ between the SND and $\sd$ distribution functions switches its monotonicity pattern just once --- from increase to decrease --- on the interval $[0,\infty)$, which provides a more manageable expression for the Kolmogorov distance $d_\ko(p)$. 

The next step concerns the difficulty that the expression \eqref{eq:f_p} for $f_p(x)$ contains the so-called Wallis ratio $\Ga(\frac{p+1}{2})/\Ga(\frac{p}{2})$, which is not algebraic, and whose logarithm or derivative or logarithmic derivative is not algebraic either. 
To deal with this problem, we have just developed in \cite{wallis-ratio} series of high-precision upper and lower algebraic bounds on the Wallis ratio; for the purposes of the present paper, the first upper bound and the second lower bound in the corresponding series in \cite{wallis-ratio} already suffice (Lemma~\ref{lem:wallis}). \big(A recent paper \cite{mortici} provided other new upper and lower bounds on the Wallis ratio, improving on a number of preceding results. The series of bounds given in \cite{wallis-ratio} (except a few first members of those series) are tighter than all the bounds in \cite{mortici}.\big) 
By using the mentioned results of \cite{wallis-ratio}, we obtain an upper bound, written as $H(a,x)/\sqrt{2\pi}$, on the difference $f_\infty(x)-f_p(x)$ between the densities of the SND and $\sd$, which has an algebraic expression in place of the Wallis ratio (Lemma~\ref{lem:H}). 

According to \eqref{eq:d=}, $d_\ko(p)$ equals a definite integral (in $x$) of the difference $f_\infty(x)-f_p(x)$; so, this integral can be bounded from above by the corresponding integral of the just mentioned upper bound $H(a,x)/\sqrt{2\pi}$. 
However, the latter integral is still problematic to estimate accurately enough. Toward that end, by some tweaking of the third-order Taylor polynomial in $a$ for $H(a,x)$ near $a=0$, we construct an upper bound $\tH_2(a,x)$ on $H(a,x)$, which is just the product of $\vpi(x)$ and a polynomial in $a,x$  (Lemma~\ref{lem:H<tH}). 
Thus, the bound $\tH_2(a,x)$ has certain nice properties (Lemma~\ref{lem:xxa}). Also, the relevant definite integral of $\tH_2(a,x)$ (corresponding to the mentioned one of $H(a,x)$) can be easily expressed in terms of the functions $\vpi$ and $\Phi$, thus finally resulting in the bound $B(a,\tx_a)$ in \eqref{eq:<B}. 

Inequality \eqref{eq:B<C/p} (which, together with \eqref{eq:<B}, yields the inequality in \eqref{eq:<C/p})  
is provided by Lemma~\ref{lem:B<C/p}, 
whose proof is rather technical and relies on the Mathematica command  \verb!Reduce!, as described above. 
As for Proposition~\ref{prop:}, it follows easily from Lemma~\ref{lem:x_p} and the result of \cite{mono-student}.  

It appears that essentially the same method can be used to obtain even tighter upper (as well as lower) bounds on the distances $d_\ko(p)$ and $d_\tv(p)$; toward such an end, one could use bounds in \cite{wallis-ratio} on the Wallis ratio of higher orders of accuracy, as well as tweaked-Taylor polynomials for $H(a,x)$ of higher orders. The limitations on the attainable accuracy of such bounds on $d_\ko(p)$ appear to be mainly set by the existing computational power; also, the proofs of the yet tighter bounds can be expected to be even more complicated. 

In accordance with the above description of the scheme of proof, the current section is organized as follows. In Subsection~\ref{lemmas}, the mentioned lemmas are stated, thus presenting most of the main steps of proof. Next, in the same subsection, Proposition~\ref{prop:} and Theorems~\ref{th:bound} and \ref{th:1/4} are proved based on these lemmas. 
Finally, in Subsection~\ref{proofs-lemmas} the lemmas stated in Subsection~\ref{lemmas} (and requiring proof) are proved. 

\subsection{Statements of lemmas, and proofs of the main results}\label{lemmas}
Introduce 
\begin{equation}\label{eq:r_p}
	r_{p,q}(x):=\frac{f_p(x)}{f_q(x)}\quad\text{and}\quad 
	r_p(x):=r_{p,\infty}(x)=\frac{f_p(x)}{\vpi(x)}. 
\end{equation}

\begin{lemma}\label{lem:x_p} 
For each pair $(p,q)$ such that $0<p<q\le\infty$
\begin{enumerate}[(i)]
	\item 
	the ratio $r_{p,q}(x)$ 
decreases in $x\in[0,1]$ from $r_{p,q}(0)<1$, and then increases in $x\in[1,\infty)$ to $\infty$; therefore, 
\item there is a unique point $x_{p,q}\in(0,\infty)$ (which is in fact greater than $1$) such that 
\begin{equation}\label{eq:x_p}
	\begin{aligned}
	&\text{$f_p(x)<f_q(x)$ for all $x\in[0,x_{p,q})$,} \\ 
	&\text{$f_p(x_{p,q})=f_q(x_{p,q})$,} \\
	&\text{$f_p(x)>f_q(x)$ for all $x\in(x_{p,q},\infty)$,} 	
	\end{aligned}
\end{equation}
and hence 
\begin{equation}\label{eq:d=}
	d_\ko(p,q)=F_q(x_{p,q})-F_p(x_{p,q}). 
\end{equation}
\end{enumerate}
\end{lemma}

For brevity, let 
\begin{equation}\label{eq:x_p:=}
	x_p:=x_{p,\infty}. 
\end{equation}

\begin{lemma}\label{lem:wallis}
For all real $p>0$ 
\begin{equation}\label{eq:wallis}
	L_2(a)<r_p(0)<U_1(a), 
\end{equation}
where 
\begin{equation*}
	L_2(a):=\frac{(1 + 2 a)^{1/2}}{(1 + a)^{7/8} (1 + 3 a)^{1/8}} \quad\text{and}\quad 
	U_1(a):=\frac1{(1 + a)^{1/4}}. 
\end{equation*}
\end{lemma}

This follows 
by the main result in \cite{wallis-ratio}; the notations $r_p(0)$, $L_k(a)$, and $U_k(a)$ in the above  Lemma~\ref{lem:wallis} correspond to $r(p)$, $L_k(p)$, and $U_k(p)$ in \cite{wallis-ratio}. 

The first inequality in \eqref{eq:wallis}, together with the definition \eqref{eq:f_p},  immediately yields 

\begin{lemma}\label{lem:H}
For all real $p>0$ and all $x\in\R$ 
\begin{equation*}
	f_\infty(x)-f_p(x)<\frac{H(a,x)}{\sqrt{2\pi}},
\end{equation*}
where
\begin{equation*}
	H(a,x):=e^{-x^2/2}-L_2(a)(1+ax^2)^{-\frac{1+a}{2a}}. 
\end{equation*}
\end{lemma}

By some tweaking of the third-order Taylor polynomial in $a$ for $H(a,x)$ near $a=0$, one obtains 
\begin{equation}\label{eq:tH2}
	\tH_2(a,x):=\frac{aP(a,x)}{384} e^{-x^2/2},
\end{equation}
where $P(a,x)$ is as in \eqref{eq:P},  
so that $\tH_2(a,x)$ be an upper bound on $H(a,x)$: 

\begin{lemma}\label{lem:H<tH}
For all $a\in(0,\frac{29}{50}]$ and $x\in(0,\frac{123}{50})$ 
\begin{equation*}
	H(a,x)<\tH_2(a,x). 
\end{equation*}
\end{lemma}

\begin{lemma}\label{lem:xxa}
For each $a\in(0,1)$, there is a unique real root $x>0$ of the polynomial equation 
\eqref{eq:P}, so that $\tx_a$ is correctly defined in the statement of Theorem~\ref{th:bound}. 
Moreover, 
\begin{equation*}
\begin{aligned}
	&\text{$\tH_2(a,x)>0$ for all $x\in(0,\tx_a)$,} \\
	&\text{$\tH_2(a,\tx_a)=0$,} \\
	&\text{$\tH_2(a,x)<0$ for all $x>\tx_a$.} \\
\end{aligned} 
\end{equation*}
Furthermore, $\tx_a$ is strictly and continuously increasing in $a\in(0,1)$. 
\end{lemma} 

\begin{lemma}\label{lem:B<C/p}
For all $p\ge4$ inequality \eqref{eq:B<C/p} holds. 
\end{lemma}


Now one is ready to prove Proposition~\ref{prop:} and Theorems~\ref{th:bound} and \ref{th:1/4}, which will be done in this order. 

\begin{proof}[Proof of Proposition~\ref{prop:}.] 
Take indeed any $p$ and $q$ such that $0<p<q\le\infty$. 
By 
Lemma~\ref{lem:x_p} and the symmetry of the $\sd$, 
\begin{multline*}
	\quad d_\tv(p,q)=\int_0^{x_{p,q}}(f_q-f_p)+\int_{x_{p,q}}^\infty(f_p-f_q) \\ 
	=2\int_0^{x_{p,q}}(f_q-f_p)=2\big(F_q(x_{p,q})-F_p(x_{p,q})\big)=2d_\ko(p,q),  
\end{multline*}
which proves part (i) of Proposition~\ref{prop:}. 
Part (ii) of the proposition now follows by the second equality in \eqref{eq:ko,tv} and the stochastic monotonicity result of \cite{mono-student}, which implies that $F_p(x)$ is strictly increasing in $p\in(0,\infty]$ for each $x\in(0,\infty)$. 
\end{proof}

\begin{proof}[Proof of Theorem~\ref{th:bound}] 
The equality in \eqref{eq:<B} immediately follows from Proposition~\ref{prop:}. 
Take any $a\in(0,\frac{29}{50}]$ (corresponding to $p\ge\frac{50}{29}$). 
We claim that $x_p<\tx_a$, where $x_p$ and $\tx_a$ are as in \eqref{eq:x_p:=} 
and Lemma~\ref{lem:xxa}, respectively. 
Assume the contrary, that $x_p\ge\tx_a$. 
Note that $\tH_2(\frac{29}{50},\frac{123}{50})<0$; so, by Lemma~\ref{lem:xxa}, $\tx_{29/50}<\frac{123}{50}$ and hence $\tx_a<\frac{123}{50}$ for all $a\in(0,\frac{29}{50}]$. 
Therefore, in view of Lemmas~\ref{lem:H} and \ref{lem:H<tH}, 
\begin{equation}\label{eq:< <}
	\sqrt{2\pi}\big(f_\infty(x)-f_p(x)\big)<H(a,x)<\tH_2(a,x)  
\end{equation}
for all $x\in(0,\tx_a]$ --- still assuming that $a\in(0,\frac{29}{50}]$. 
On the other hand, by Lemma~\ref{lem:x_p}, $0\le f_\infty(x)-f_p(x)$ for all $x\in(0,x_p]$ and hence, by the assumption, for all $x\in(0,\tx_a]$. 
Now \eqref{eq:< <} implies $0<\tH_2(a,\tx_a)$, which contradicts Lemma~\ref{lem:xxa}. 
Thus, indeed $x_p<\tx_a$. 
Recalling now \eqref{eq:d=} and using \eqref{eq:< <} and (again) Lemma~\ref{lem:xxa}, and also recalling \eqref{eq:tH2}, one has 
\begin{align*}
	d_\ko(p)=F_\infty(x_p)-F_p(x_p)
	&=\int_0^{x_p}\big(f_\infty(x)-f_p(x)\big)\dd x \\
	&<\int_0^{x_p}\frac{\tH_2(x)}{\sqrt{2\pi}}\dd x 
	<\int_0^{\tx_a}\frac{\tH_2(x)}{\sqrt{2\pi}}\dd x.  
\end{align*}
It remains to verify that $\int_0^x\frac{\tH_2(u)}{\sqrt{2\pi}}\dd u=B(a,x)$, which can be done either by hand or using Mathematica. 
The proof of Theorem~\ref{th:bound} is now complete, modulo the lemmas. 
\end{proof}

\begin{proof}[Proof of Theorem~\ref{th:1/4}] 
The relations in \eqref{eq:<C/p} immediately follow by Theorem~\ref{th:bound} and Lemma~\ref{lem:B<C/p}. 
It remains to verify \eqref{eq:lim=C}. 
First here, use l'Hospital's rule to find that for all real $x$ 
\begin{equation}\label{eq:lim}
	\lim_{a\downarrow0}\frac{f_{1/a}(x)-f_\infty(x)}a	
	=\lim_{a\downarrow0}\pd{f_{1/a}(x)}a=\la(x):=\frac{x^4-2 x^2-1}4\,\vpi(x);
\end{equation}
the second equality in \eqref{eq:lim} can be obtained either using the Mathematica commands \verb!D! (for differentiation), \verb!Simplify!, and \verb!Limit! or otherwise. 

Next, introduce 
\begin{equation}\label{eq:ca,ga}
	c_a:=f_{1/a}(0)\quad\text{and}\quad g_a(x)=f_{1/a}(x)/c_a 
\end{equation}
for all real $a\ge0$,  
assuming the convention $1/0:=\infty$, so that $f_{1/a}(x)=c_a g_a(x)$. 
Then for all real $a\ge0$ and all real $x$ 
\begin{equation}\label{eq:dif}
	|f_{1/a}(x)-f_\infty(x)|
	\le|c_a-c_0|g_a(x)+c_0|g_a(x)-g_0(x)|
	\le|c_a-c_0|+|g_a(x)-g_0(x)|, 
\end{equation}
since $g_a(x)\le1$ and $c_0=1/\sqrt{2\pi}<1$. 
By \eqref{eq:lim} and \eqref{eq:ca,ga}, the ratio $\frac{|c_a-c_0|}a$ is continuous in $a>0$ and converges to a finite limit ($\vpi(0)/4$) as $a\downarrow0$, and hence is bounded in $a\in(0,1]$. 
Now note that 
\begin{equation*}
	\Big|\pd{g_a(x)}a\Big|=(1 + a x^2)^{-(1 + 3 a)/(2 a)}\,|(Dg)(a,x)| 
	\le|(Dg)(a,x)|,
\end{equation*}
where
\begin{equation*}
	(Dg)(a,x):=\frac{\left(1+a x^2\right) \ln \left(1+a x^2\right)-a (1+a) x^2}{2 a^2}. 
\end{equation*}
Using the Taylor expansion $\ln(1+u)=u-\th u^2/2$ for $u>0$ and some $\th=\th(u)\in(0,1)$, one sees that $(Dg)(a,x)$ is a polynomial in $a,x,\th$ and hence bounded in $(a,x)\in(0,1]\times[0,\tx_0]$ --- note that, in accordance with the definition of $\tx_a$ in Theorem~\ref{th:bound}, 
$$\tx_0=\sqrt{1 + \sqrt2}\in(0,\infty);$$ 
hence, $\Big|\pd{g_a(x)}a\Big|$ is bounded in $(a,x)\in(0,1]\times[0,\tx_0]$ and, by the mean value theorem, so is $\frac{|g_a(x)-g_0(x)|}a$.  
Recalling also \eqref{eq:dif} and that the ratio $\frac{|c_a-c_0|}a$ is bounded in $a\in(0,1]$, one concludes that 
the ratio $\frac{|f_{1/a}(x)-f_\infty(x)|}a$ is bounded in $(a,x)\in(0,1]\times[0,\tx_0]$. 
So, by \eqref{eq:lim} and dominated convergence, 
\begin{align*}
	p\,d_\ko(p)\ge p\,\big[F_\infty(\tx_0)-F_p(\tx_0)\big]
	&=-\int_0^{\tx_0}\frac{f_{1/a}(x)-f_\infty(x)}a \dd x \\
	&\underset{a\downarrow0}\longrightarrow
	-\int_0^{\tx_0}\la(x) \dd x
	=\frac{(\tx_0^3+\tx_0)\,\vpi(\tx_0)}4=C,  
\end{align*}
where $\la(x)$ is defined in \eqref{eq:lim}. 
This, together with \eqref{eq:<C/p}, implies \eqref{eq:lim=C}. 
The proof of Theorem~\ref{th:1/4} is now complete, modulo the lemmas. 
\end{proof}

\subsection{Proofs of the lemmas}\label{proofs-lemmas}

\begin{proof}[Proof of Lemma~\ref{lem:x_p}] 
Take indeed any $p$ and $q$ such that $0<p<q\le\infty$.  
A key observation here (borrowed from \cite{mono-student}) is that 
$r_{p,q}(x)$ 
decreases in $x\in[0,1]$ and increases in $x\in[1,\infty)$. 
Moreover, 
by the lemma in \cite{mono-student}, $f_p(0)$ increases in $p>0$ and hence 
$r_{p,q}(0)<1$. 
On the other hand, it is easy to see that $r_{p,q}(x)\to\infty$ as $x\to\infty$. 
This completes the proof of part (i) of Lemma~\ref{lem:x_p}, 
which in turn implies that there is a unique $x_{p,q}\in(0,\infty)$ such that 
$r_{p,q}(x)<1$ for $x\in[0,x_{p,q})$, $r_{p,q}(x_{p,q})=1$, and $r_{p,q}(x)>1$ for $x\in(x_{p,q},\infty)$ (at that necessarily $x_{p,q}>1$). 
In other words, one has the relations \eqref{eq:x_p}. 
Since $(F_q-F_p)'=f_q-f_p$, one now sees that $F_q(x)-F_p(x)$ increases in $x\in[0,x_{p,q}]$ from $0$ to $F_q(x_{p,q})-F_p(x_{p,q})>0$, and then decreases in $x\in[x_{p,q},\infty)$ to $0$. 
So, \eqref{eq:d=} follows by the symmetry of the Student and standard normal distributions. 
Thus, the lemma is completely proved. 
\end{proof}

\begin{proof}[Proof of Lemma~\ref{lem:H<tH}]
Indeed assume that $a\in(0,\frac{29}{50}]$ and $x\in(0,\frac{123}{50})$. 
Consider the difference 
\begin{equation*}
	\tilde\de(a):=\tilde\de(a,x):=H(a,x)-\tH_2(a,x)
	=\frac{\tP(a,x)}{384 e^{x^2/2}} 
	- L_2(a)(1+ax^2)^{-\frac{1+a}{2a}},  
\end{equation*}
where 
\begin{equation*}
	\tP(a,x):=384-aP(a,x). 
\end{equation*}
We have to show that $\tilde\de(a,x)<0$. Obviously, the system of inequalities $\tP(a,x)\le0$, $0<a\le\frac{29}{50}$, and $0<x<\frac{123}{50}$ is algebraic and thus, by the well-known result of Tarski \cite{tarski48} can be solved completely algorithmically. 
The 
Mathematica command 
\verb!Reduce[tP<=0 && 29/50>=a>0 && 123/50>x>0]! outputs \verb!False!, 
where \verb!tP! stands for $\tP(a,x)$. 
This means that $\tP(a,x)>0$ --- for all $a\in(0,\frac{29}{50}]$ and $x\in(0,\frac{123}{50})$. 
So, 
$\tilde\de(a,x)$ equals 
\begin{equation*}
	\de(a):=\de(a,x):=
	\ln\frac{\tP(a,x)}{384 e^{x^2/2}} 
	- \ln\Big(L_2(a)(1+ax^2)^{-\frac{1+a}{2a}}\Big)  
\end{equation*}
in sign. 
Introduce 
\begin{align*}
	(D\de)(a)&:=4a^2\de'(a)
	=\frac{4 a Q(a,x)}{\tP(a,x)}-\frac{2 (1+a)}{1+ax^2}-2 \ln \left(1+ax^2\right) \\
	&\qquad\qquad\qquad+\frac{1}{6} \left(72
   a+\frac{21}{1+a}-\frac{6}{1+2a}+\frac{1}{1+3a}-4\right),
\\
   (DD\de)(a)&:=
   \frac{(D\de)'(a)}{2 a^3}\, 
   (1+a)^2 (1+2a)^2 (1+3a)^2  \left(1+ax^2\right)^2\,\tP(a,x)^2, 
\end{align*}
where 
\begin{equation*}
	Q(a,x):=a^3 \left(20 x^4-60x^2-33\right)-96 a \left(x^4-2 x^2-1\right)-768. 
\end{equation*}
Note that $(DD\de)(a)$ is a polynomial in $a$ and $x$, of degree $11$ in $a$ and of degree $20$ in $x$. 
The command 
\verb!Reduce[DDde>=0 && 29/50>=a>0 && 123/50>x>0]! outputs \verb!False!, 
where \verb!DDde! stands for $(DD\de)(a)$. 
This means that $(DD\de)(a)<0$ --- for all $a\in(0,\frac{29}{50}]$ and $x\in(0,\frac{123}{50})$.
On the other hand, one can check (using Mathematica or otherwise) that $(D\de)(0+)=\de(0+)=0$.  
Thus, one concludes that indeed $\de(a,x)<0$ and hence $\tilde\de(a,x)<0$. 
\end{proof}

\begin{proof}[Proof of Lemma~\ref{lem:xxa}] 
Take indeed any $a\in(0,1)$. 
By \eqref{eq:tH2}, $\tH_2(a,x)$ equals $P(a,x)$ in sign. 
So, the first two sentences of Lemma~\ref{lem:xxa} can be proved using 
the Mathematica command 
\verb!Reduce[P>0! 
\verb!&& 0<a<1 && x>0,x]!.   
That $\tx_a$ is strictly increasing in $a\in(0,1)$ now follows by 
the command 
\verb!Reduce[PP[a,x]==0! \break
\verb!==PP[b,y] && 0<a<b<1 && 0<y<=x]!,
which (takes about 15 seconds on a standard laptop and) outputs \verb!False!; here,  
\verb!PP[a,x]! 
stands for $P(a,x)$. 
Finally, the continuity of $\tx_a$  in $a$ can be verified by the implicit function theorem; here, it is enough to check that $\pd Px(a,\tx_a)\ne0$ for all $a\in(0,1)$, 
which can be done using the command 
\verb!Reduce[P==0 && DPx==0 && 0<a<1 && x>0]! (with \verb!DPx! standing for $\pd Px(a,x)$), 
which outputs \verb!False!. 
\end{proof}

\begin{proof}[Proof of Lemma~\ref{lem:B<C/p}]
By Lemma~\ref{lem:xxa}, $a\mapsto\tx_a$ is a one-to-one map of $(0,\frac14]$ onto $(\tx_0,\tx_{1/4}]$. 
Let $(\tx_0,\tx_{1/4}]\ni x\mapsto a_x\in(0,\frac14]$ be the corresponding inverse map. 
So, it suffices to 
show that $B(a_x,x)<Ca_x$ for all $x\in(\tx_0,\tx_{1/4}]$. 
Assume indeed that $x\in(\tx_0,\tx_{1/4}]$ and consider the ratio 
\begin{equation}\label{eq:rho}
	\rho(x):=\frac{B(a_x,x)-Ca_x}{a_x^3}.   
\end{equation}
Introduce also 
\begin{align*}
	q_1(x):=&3 + 12 x^2 + 30 x^4 - 28 x^6 + 3 x^8, \\ 
	q_2(x):=&-783 - 2952 x^2 - 1284 x^4 + 2952 x^6 - 234 x^8 - 1608 x^{10}  \\
 &+ 964 x^{12}- 168 x^{14} + 9 x^{16}, \\ 
	q_3(x):=&33 + 60 x^2 - 20 x^4. 
\end{align*}
The command \verb!Reduce[q3 <= 0 && xxa[0] < x <= xxa[1/4]]! \big(with \verb!q3! and \verb!xxa[a]! standing for $q_3(x)$ and $\tx(a)$\big)
outputs \verb!False!, which shows that $q_3(x)>0$. 
Now using the command 
\verb!Reduce[P==0 && 0<a<=1/4 && xxa[0]<x<=xxa[1/4]]!,  
where $P$ stands again for the polynomial $P(a,x)$ as in \eqref{eq:P}, 
one finds that 
\begin{equation*}
	a_x=2\,\frac{q_1(x)+\sqrt{q_2(x)}}{q_3(x)};   
\end{equation*}
moreover, $a_x>0$ and $q_3(x)>0$ imply that $q_1(x)+\sqrt{q_2(x)}>0$. 
So, in view of \eqref{eq:rho}, 
\begin{multline*}
	\rho'(x)\frac{e^{x^2/2}}x\,  24 \sqrt{\pi} \sqrt{q_2(x)} \big(q_1(x) + \sqrt{q_2(x)}\big)^3\\ 
	 = \sqrt{2} \big(p_{00}(x) + p_{01}(x)\sqrt{q_2(x)}\,\big) +
  e^{x^2/2}\, C\sqrt\pi \big(p_{10}(x) + p_{11}(x)\sqrt{q_2(x)}\,\big), 
\end{multline*}
where 
\begin{align*}
	p_{00}(x) :=& 105705 x - 1945539 x^3 - 13305006 x^5 - 
    26650971 x^7 - 3174714 x^9\\ 
    &+ 49512627 x^{11} + 23388786 x^{13} - 
    45078003 x^{15} - 9879213 x^{17} \\
    &+ 26892909 x^{19} - 5379786 x^{21} - 
    8094383 x^{23} + 6008972 x^{25} \\
    &- 1844301 x^{27} + 296622 x^{29} - 
    24435 x^{31} + 810 x^{33}, \\
	p_{01}(x) :=& -21789 x - 259929 x^3 - 492804 x^5 + 366741 x^7 + 
    967263 x^9 \\
    &- 120468 x^{11} - 487080 x^{13} + 188214 x^{15} + 
    177266 x^{17} - 151973 x^{19}\\
    & + 43674 x^{21} - 5625 x^{23} + 270 x^{25}, \\
	p_{10}(x) :=& -1368576 + 9287136 x^2 + 67830048 x^4 + 
    113324832 x^6 - 52129440 x^8\\
    & - 230541408 x^{10} + 74263392 x^{12} + 
    151161696 x^{14} - 110996640 x^{16} \\
    &+ 30085440 x^{18} - 3715200 x^{20} + 
    172800 x^{22}, \\
	p_{11}(x) :=& 171072 + 1974240 x^2 + 2340576 x^4 - 4409568 x^6 - 
    3045600 x^8 \\
    &+ 2911680 x^{10} - 700800 x^{12} + 57600 x^{14}. 
\end{align*}
Executing now the command 
\verb!Reduce[z<=0 && xxa[0]<x<=xxa[1/4]]! with \verb!z! standing for $p_{10}(x) + p_{11}(x)\sqrt{q_2(x)}$, one sees that $p_{10}(x) + p_{11}(x)\sqrt{q_2(x)}>0$, so that $\rho'(x)$ equals 
\begin{equation*}
	\rho_1(x):=\frac{\sqrt{2} \big(p_{00}(x) + p_{01}(x)\sqrt{q_2(x)}\big)}
	{e^{x^2/2}\, \big(p_{10}(x) + p_{11}(x)\sqrt{q_2(x)}\big)}
	+C\sqrt\pi 
\end{equation*}
in sign. 


Next, 
\begin{align*}
	\rho_2(x):=&\rho_1'(x)\, e^{x^2/2}\, (p_{10}(x) +p_{11}(x) \sqrt{q_2(x)})^2/\sqrt{2} \\
	=&c_0(x) + c_1(x) \sqrt{q_2(x)} + c_2(x) q_2(x) + c_3(x)/\sqrt{q_2(x)}, 
\end{align*}
where 
\begin{align*}
	c_0(x):=&p_{10}(x) p_{00}'(x)-p_{00}(x) p_{10}'(x)-x p_{00}(x)
   p_{10}(x), \\
	c_1(x):=&p_{11}(x) p_{00}'(x)- p_{00}(x) p_{11}'(x)- x p_{00}(x)
   p_{11}(x)+ p_{10}(x) p_{01}'(x) \\
   &- p_{01}(x) p_{10}'(x)- x
   p_{01}(x) p_{10}(x), \\
	c_2(x):=&p_{11}(x) p_{01}'(x)-p_{01}(x) p_{11}'(x)-x p_{01}(x)
   p_{11}(x), \\
	c_3(x):=&q_2'(x) \big(p_{01}(x) p_{10}(x)-p_{00}(x) p_{11}(x)\big)/2. 
\end{align*}
The command \verb!Reduce[rho2[x]<=0 && xxa[0]<x<=xxa[1/4]]! outputs \verb!False!. 
So, $\rho_2(x)>0$ (for all $x\in(\tx_0,\tx_{1/4}]$) and hence $\rho_1(x)$ increases in such $x$.  
Moreover, $\rho_1(\tx_0)=0$, which implies that $\rho_1>0$. That is, $\rho'>0$ and $\rho$ is increasing on the interval $(\tx_0,\tx_{1/4}]$, to $\rho(\tx_{1/4})<0$. 
Thus, $\rho<0$ on $(\tx_0,\tx_{1/4}]$, which 
implies that indeed $B(a_x,x)<Ca_x$ for all $x\in(\tx_0,\tx_{1/4}]$. 
\end{proof}


\bibliographystyle{abbrv}


\bibliography{C:/Users/Iosif/Dropbox/mtu/bib_files/citations12.13.12}

\end{document}